\documentclass[12pt]{amsart}
\usepackage{amsmath,amsthm,amsfonts,amssymb,times,latexsym,mathabx,url}

\newtheorem{theorem}{Theorem}[section]
\newtheorem{prop}[theorem]{Proposition}
\newtheorem{lem}[theorem]{Lemma}

\newtheorem{rem}[theorem]{Remark}
\numberwithin{equation}{section}
\newtheorem{conj}[theorem]{Conjecture}

\renewcommand{\L}{\mathcal{L}}
\newcommand{\e}{\varepsilon}

\renewcommand{\a}{\alpha}
\renewcommand{\b}{\beta}

\newcommand{\N}{\mathbb{N}}
\newcommand{\Z}{\mathbb{Z}}
\renewcommand{\leq}{\leqslant}
\renewcommand{\geq}{\geqslant}
\renewcommand{\d}{\delta}
\renewcommand{\L}{\Lambda}

\makeatletter
\renewcommand{\pmod}[1]{\allowbreak\mkern7mu({\operator@font mod}\,\,#1)}
\makeatother

\newcommand{\C}{\mathbb C}

\newcommand{\T}{\mathbb T}
\newcommand{\R}{\mathbb R}
\renewcommand{\L}{\Lambda}

\renewcommand{\a}{\alpha}
\renewcommand{\b}{\beta}
\newcommand{\g}{\gamma}
\renewcommand{\leq}{\leqslant}
\renewcommand{\geq}{\geqslant}

\begin{document}

\title[Trigonometric polynomials with frequencies in the set of squares]{%Conjecture of Cilleruelo and Cordoba \\
Trigonometric polynomials with frequencies \\ in the set of squares and divisors in a short interval}
\author{Mikhail R. Gabdullin}
\date{}
%\address{Lomonosov Moscow State University, Leninskie Gory str., 1, Moscow, Russia, 119991}
%\address{Steklov Mathematical Institute,
%	Gubkina str., 8, Moscow, Russia, 119991}
%\email{gabdullin.mikhail@yandex.ru}

\address{Department of mathematics, 1409 West Green Street, University of Illinois at Urbana-Champaign, Urbana, IL 61801, USA; Steklov Mathematical Institute,
	Gubkina str., 8, Moscow, 119991, Russia}
\email{gabdullin.mikhail@yandex.ru, mikhailg@illinois.edu}

\thanks{2010 Mathematics Subject Classification: Primary 42A05, 11A05}

\thanks{Keywords and phrases: squares, trigonometric polynomials, divisors}

\thanks{ORCID: 0000-0001-8774-8334}

\begin{abstract}
Let $\g_0=\frac{\sqrt5-1}{2}=0.618\ldots$ . We prove that, for any $\e>0$ and any trigonometric polynomial $f$ with frequencies in the set $\{n^2: N \leq n\leq N+N^{\g_0-\e}\}$, the inequality
$$
\|f\|_4 \ll \e^{-1/4}\|f\|_2
$$
holds, which makes a progress on a conjecture of Cilleruelo and C\'ordoba. We also present a connection between this conjecture and the conjecture of Ruzsa which asserts that, for any $\e>0$, there is $C(\e)>0$ such that each positive integer $N$ has at most $C(\e)$ divisors in the interval $[N^{1/2}, N^{1/2}+N^{1/2-\e}]$.
\end{abstract}

\date{\today}
\maketitle

\section{Introduction} 

Let us say that a trigonometric polynomial $f$ has frequencies in a set $A\subseteq\Z$ if $f(x)=\sum_{n\in A}a_ne(nx)$ with $a_n\in\C$ (here and in what follows $e(x)=e^{2\pi ix}$). Let $p>2$. Recall that a set $A\subseteq \Z$ is said to be a $\L_p$-set if there exists a constant $C(A,p)>0$ such that for any trigonometric polynomial $f$ with frequencies in the set $A$, the inequality
\begin{equation}\label{1.1}
\|f\|_p \leq C(A,p)\|f\|_2
\end{equation}
holds, where $\|f\|_p=\left(\int_{\T}|f(x)|^pdx\right)^{1/p}$ and $\T=\R/\Z$.  It is easy to show that if a set $A$ is a $\L_p$-set for some $p>2$, then the counting function of $A$ cannot grow too fast. To be more precise, let $A_N=A\cap[-N,N]$ and $f(x)=\sum_{n\in A_N}e(nx)$; then by the properties of the Dirichlet kernel $D_N(x)=\sum_{|k|\leq N}e(kx)$, H\"older's inequality and (\ref{1.1}),
$$
|A_N|=f(0)=(f*D_N)(0)=\int_{\T}f(-y)D_N(y)dy \leq \|f\|_p\|D_N\|_q \ll |A_N|^{1/2}N^{1/p}
$$
(here $q$ is defined by $1/p+1/q=1$), and, hence, $|A_N| \ll N^{2/p}$. Using probabilistic methods, Bourgain \cite{Bour} showed that this upper bound is tight: for instance, Theorem 3 in \cite{Bour} states that for any $p>2$, there exists a $\L_p$-subset $A$ of the set of prime numbers of maximal density, that is, with $\varliminf_{N\to\infty}|A_N|N^{-2/p}\gg1$. However, that proof does not produce any examples of $\L_p$-sets. On the other hand, we note that there is the classical result that any lacunary sequence is a $\L_p$-set for every $p>2$ (e.g., \cite{Gr}, Theorem 3.6.4). 

In this regard, it is of a great interest to find polynomial sequences which are $\L_p$-sets. There is a famous and still unsolved conjecture about the set of squares $\{n^2: n\in \N\}$, which was discussed by W.~Rudin (see the end of section 4.6 in \cite{Rud}).

\begin{conj}\label{conj1.1}
The set of squares is a $\L_p$-set for any $2<p<4$. 
\end{conj}

A.~C\'ordoba \cite{Cor} proved the inequality (\ref{1.1}) with any fixed $2<p<4$ for polynomials $f(x)=\sum_{n=1}^Na_ne(n^2x)$ with positive and nonincreasing coefficients $a_n$. However, the general case of Conjecture \ref{conj1.1} seems to be beyond the reach of current methods. We mention only that, for $2<p<4$,
$$
\left\| \sum_{n\leq N}e(n^2x) \right\|_p \asymp N^{1/2};
$$
it follows from the distribution estimate
$$
\mu\left\{x\in\T: \left|\sum_{N\leq n\leq 2N} e(n^2x) \right| \geq \a>0 \right\} \ll \frac{N^2}{\a^4},
$$
obtained in \cite{JH}.
 
On the other hand, it is well-known (e.g., see the inequality (1.7) in \cite{Bour}) that 
\begin{equation}\label{1.2}
\left\|\sum_{n\leq N}e(n^2x)\right\|_4 \asymp N^{1/2}(\log N)^{1/4},
\end{equation}
so the set of squares is not a $\L_4$-set. Nevertheless, the case $p=4$ is of great interest, in particular because the $L_4$-norms are associated with the number of solutions of certain equations: for $f(x)=\sum_{n\in A}e(nx)$, where $A$ is a finite set, $\|f\|_4^4$ is nothing but the additive energy $E(A)$ of the set $A$, which is defined by 
$$
E(A)=\#\{(a_1,b_1,a_2,b_2)\in A^4: a_1+b_1=a_2+b_2 \}.
$$ 
Note that the trivial lower bound  $\|f\|_4\geq \|f\|_2=|A|^{1/2}$ for $L_4$-norm resembles the trivial lower bound $E(A)\geq 2|A|^2-|A|$ for the additive energy (which corresponds to the trivial solutions, that is, ones of the form $\{a_1,b_1\}=\{a_2,b_2\}$). Recall that the sets $A$ with $E(A)=2|A|^2-|A|$ (or, in other words, the sets for which the equation $a_1+b_1=a_2+b_2$ has only trivial solutions) are called Sidon sets. It is easy to show (and we will see it in Section \ref{sec2}) that any Sidon set is a $\L_4$-set.

J.~Bourgain \cite{Bour} made the following conjecture related to the case $p=4$.

\begin{conj}\label{conj1.2}
There exists $\d>0$ such that, for any $N\geq2$ and any trigonometric polynomial $f$ with frequencies in the set $\{n^2: n\leq N\}$,
$$
\|f\|_4 \ll \|f\|_2 \cdot (\log N)^{\delta}.
$$	
\end{conj}

Note that $\d$ must be at least $1/4$ because of (\ref{1.2}). 

This quite natural statement is probably hard to prove. It can also be shown that it is enough to verify this bound for the case where the coefficients of $f$ are equal to zero or one, that is, to prove that for any $A\subseteq \{n^2: n\leq N\}$ one has $E(A)\leq |A|^2(\log N)^{O(1)}$; see \cite{CG} for the details (Theorem 14) and connections to other problems in the area. 

In this paper we will focus on the following conjecture of J.~Cilleruelo and A.~C\'ordoba (see \cite{CC}), which can be thought of as a ``truncated'' version of Bourgain's one.

\begin{conj}\label{conj1.3}
For any $\g\in(0,1)$ there is $c(\g)>0$ such that, for any trigonometric polynomial $f$ with frequencies in the set $\{n^2: N \leq n\leq N+N^{\g}\}$, we have
$$
\|f\|_4 \leq  c(\g) \|f\|_2 .
$$
\end{conj}

This can be easily proved for $\g\leq1/2$, since the set  $\{n^2: N \leq n\leq N+2\sqrt2N^{1/2}\}$ is a Sidon set (see the Appendix for a short proof); however, it seems that it has been open for any $\g>1/2$.

Let $\g_0=\frac{\sqrt5-1}{2}=0.618\ldots$ . Our main result is the following.

\begin{theorem}\label{th1.4}
For any $\e>0$ and any trigonometric polynomial $f$ with frequencies in the set $\{n^2: N \leq n\leq N+N^{\g_0-\e}\}$, 
$$
\|f\|_4 \ll \e^{-1/4}\|f\|_2.
$$
\end{theorem}

We reduce the proof of Theorem \ref{th1.4} to obtaining upper bounds for the number of divisors of a positive integer in a short interval. To be more precise, we make the following conjecture (here $\tau(n;a,b)=\#\{d|n: a\leq d\leq b \}$ denotes the numbers of divisors of $n$ lying in the interval $[a;b]$).

\begin{conj}\label{conj1.5}
For any $\g\in(0,1)$ there exists $C(\g)>0$ such that, for all $m\leq 3N^{1+\g}$, 
$$
\tau(m; 2N, 2N+2N^{\g})\leq C(\g).
$$
\end{conj}

Conjecture \ref{conj1.5} with a fixed $\g$ implies Conjecture \ref{conj1.3} with the same exponent $\g$. We prove Conjecture \ref{conj1.3} for any $\g<\g_0$.

\begin{theorem}\label{th1.6}
Let $\e>0$ and $k=N^{\g_0-\e}$. Then for any $m\leq 3Nk$,
$$
\tau(m; 2N, 2N+2k) \ll \e^{-1}.
$$
\end{theorem}

As we noted, this theorem implies Theorem \ref{th1.4}.

Our Conjecture \ref{conj1.5} is very close to the following conjecture of I.~Ruzsa, which was mentioned in the papers \cite{CJ} and \cite{ER}.

\begin{conj}\label{conj1.7}
For any $\e>0$ there is $C_1(\e)>0$ such that, for any positive integer $N$, 
$$
\tau(N; N^{1/2}, N^{1/2}+N^{1/2-\e}) \leq C_1(\e).
$$ 
\end{conj}

Conjectures \ref{conj1.5} and \ref{conj1.7} look very similar, and probably a progress in one of them would imply a progress in the other. Conjecture \ref{conj1.7} is known to be true for $\e>1/4$: it can be shown by the same idea we use in the proof of Theorem \ref{th1.6}, and also follows from Corollary 3.8 in \cite{CT}.  

Our method can also be used for proving analogues of Theorem \ref{th1.4} for cubes and higher powers. However, it is known that $\|\sum_{n=1}^Ne(n^3x)\|_4\asymp N^{1/2}$: this inequality immediately follows from the result of C.~Hooley \cite{Hoo} (it is also shown in the proof of Theorem 2.6 in \cite{Nat} that the number of nontrivial solutions of $a_1^3+b_1^3=a_2^3+b_2^3$ in $0\leq a_i,b_i\leq N$ is $O(N^{5/3+\e})$ for any $\e>0$; we refer the interested reader to the work \cite{Woo} for an overview of bounds for cubic exponential sums). Thus, unlike the case of squares, it is logical to conjecture that the set of cubes $\{n^3: n\in\N \}$ is a $\L_4$-set and it is natural to expect it for higher powers too. Note also that there is a conjecture of P.~Erd\H{o}s which asserts that the set of fifth powers $\{n^5: n\in\N \}$ is a Sidon set (again, probably the same holds for higher powers as well). We do not address any of these problems in this paper.

\medskip 

\textbf{Notation.} We use Vinogradov's $\ll$ notation: both $F\ll G$ and $F=O(G)$ mean that there exists a constant $C>0$ such that $|F|\leq CG$. We write $F\asymp G$ if $G\ll F\ll G$. We also use $(a,b)$ to denote the greatest common divisors of $a$ and $b$, and $[d_1,...,d_r]$ for the least common multiple of $d_1,...,d_r$.

\medskip

\textbf{Acknowledgements.} This research was carried out at Lomonosov Moscow State University with the financial support of the Russian Science Foundation (grant no. 22-11-00129).

\section{Proof of Conjecture \ref{conj1.3} under Conjecture \ref{conj1.5} } \label{sec2}

We begin with the following general estimate.

\begin{lem}\label{lem2.1}
Let $A$ be a finite set of integers with $|A|\geq2$ and
$$
r_A^-(m)=\#\{(n_1,n_2)\in A\times A: n_1-n_2=m\}.
$$
Then for $f(x)=\sum_{n\in A}a_n e(nx)$, we have
\begin{equation}\label{2.1}
\|f\|_4 \leq \left(1+\max_{m>0} r_A^-(m)\right)^{1/4}\|f\|_2 \, .
\end{equation}
\end{lem}

\begin{proof}
We have $|f(x)|^2=f(x)\overline{f(x)}=\sum_mc_me(mx)$, where
$$
c_m=\sum_{\substack{n_1,n_2\in A \\ n_1-n_2=m}}a_{n_1}\overline{a_{n_2}}
$$
(if a number $m$ does not have representations of the form $m=n_1-n_2$ with $n_1,n_2\in A$, then we set $c_m=0$). By the Cauchy-Schwarz inequality,
$$
|c_m|^2\leq r_A^-(m)\sum_{\substack{n_1,n_2\in A \\ n_1-n_2=m}}|a_{n_1}|^2|a_{n_2}|^2.
$$
Note also that $c_0=\sum_{n\in A}|a_n|^2=\|f\|_2^2$. Summing, we have
\begin{multline*}
\|f\|_4^4=\int_0^1\left|\sum_mc_me(mx)\right|^2dx=\sum_m|c_m|^2=|c_0|^2+\sum_{m\neq 0}|c_m|^2\\
\leq \|f\|_2^4+\sum_{m\neq 0} r_A^-(m)\sum_{\substack{n_1,n_2\in A \\ n_1-n_2=m}}|a_{n_1}|^2|a_{n_2}|^2 \leq \left(1+\max_{m\neq 0}r_A^-(m)\right)\cdot\|f\|_2^4 \, ,
\end{multline*} 
and, since $r_A^-(m)=r_A^-(-m)$,
\begin{equation*}
\|f\|_4 \leq \left(1+\max_{m>0} r_A^-(m)\right)^{1/4}\|f\|_2 \, .
\end{equation*}
This concludes the proof.
\end{proof}

\begin{rem}
Note that the same inequality can be written for an infinite set $A$, though the right-hand side may not be finite. However, if $A$ is a Sidon set, then any nonzero $m$ has at most one representation in the form $m=a_1-a_2$ with $a_i\in A$, and thus the inequality (\ref{2.1}) implies that any Sidon set is a $\L_4$-set, as was mentioned in the introduction.
\end{rem}

\begin{rem}
A more standard estimate involves the number 
$$
r^+_A(m)=\#\{(n_1,n_2)\in A\times A: n_1+n_2=m\}
$$
of representations of an integer $m$ as the sum of two elements of $A$. Using the fact that $f^2(x)=\sum_{m}c_m'e(mx)$, where 
$$
c'_m=\sum_{\substack{n_1,n_2\in A\\ n_1+n_2=m}}a_{n_1}a_{n_2},
$$ 	
and arguing as in the proof of Lemma \ref{lem2.1}, one can show that
$$
\|f\|_4 \leq \left(\max_{m\in\Z} r^+_A(m)\right)^{1/4}\|f\|_2 
$$
(which is also the inequality (6.1) in \cite{CG}). However, as we will see below, it is crucial for our argument to use the bound in terms of $r_A^-(m)$.
\end{rem}

Now let $A=\{n^2: N\leq n \leq N+k \}$ for some $1\leq k\leq N$. Suppose that some $m>0$ is represented as $m=n_1-n_2$ with $n_1,n_2\in A$. Write $n_i=(N+s_i)^2$, $i=1,2$, where $0\leq s_i\leq k$. Then
$$m=(N+s_1)^2-(N+s_2)^2=2Ns_1+s_1^2-2Ns_2-s_2^2=(s_1-s_2)(2N+s_1+s_2).
$$
It follows that $m\leq 2Nk+k^2\leq 3Nk$ and 
$$
r_A^-(m)\leq \tau(m; 2N, 2N+2k)=\#\{d|m: 2N\leq d\leq 2N+2k\}.
$$ 
Thus we can rewrite (\ref{2.1}) as 
\begin{equation}\label{2.2}
\|f\|_4 \leq \left(1+\max_{1\leq m\leq 3Nk} \tau(m; 2N, 2N+2k)\right)^{1/4}\|f\|_2 \, .
\end{equation}	
Now we let $k=N^{\g}$; then from Conjecture \ref{conj1.5} we have
$$
\|f\|_4 \leq \left(1+C(\gamma)\right)^{1/4} \|f\|_2, 
$$ 
and Conjecture \ref{conj1.3} follows with $c(\g)=(1+C(\g))^{1/4}$.

\section{Proof of Theorem \ref{th1.4}}	

Recall that $\g_0=\frac{\sqrt5-1}{2}$. We will prove the following version of Theorem \ref{th1.4}.

\begin{theorem}\label{th3.1}
Let $r\geq3$ be a positive integer and $k<c_0N^{\g_0-c_1/r}$, where $c_0$ and $c_1$ are some absolute constants. Then
$$
\max_{m\leq 3Nk}\tau(m; 2N,2N+2k) \leq r-1.
$$	
\end{theorem}

We note that Theorem \ref{th3.1} follows from Theorem 1.2 in \cite{CJ} and from Corollary 3.8 in \cite{CC}. For completeness, we provide an alternative proof, which is of independent interest.   

We will need the following lower bound for the least common multiple of a set of integers. 

\begin{lem}\label{lem3.2}
Let $r\in \N$. For any positive integers $d_1,...,d_r$ and $2\leq s\leq r$,
$$	
\left(\prod_{1\leq i_1<...<i_s\leq r}[d_{i_1},...,d_{i_s}]\right)^{1/{r \choose s}} \geq \frac{\prod_{i=1}^rd_i^{2/(c+1)}}{\prod_{1\leq i<j\leq r}(d_i,d_j)^{2/(c(c+1))}},
$$
where $c=r-s+1$.	
\end{lem}

\begin{rem}\label{rem2.2} 
% \rm 

1) The conclusion of this lemma does not hold for $s=1$. Indeed, in this case $c=r$, and taking $d_1=...=d_{r-1}=1$ and $d_r=d$, we see that $d^{1/r}<d^{2/(r+1)}$ whenever $r\geq2$.
	
2) The bound is sharp for any $r\geq s\geq 2$. One can take $d_1=...=d_{s-1}=1$ and $d_s=...=d_r=d$. Then each $[d_{i_1},...,d_{i_s}]$ is equal to $d$, and so is the left-hand side of the inequality. The right-hand side is $d^a$, where 
$$
a=\frac{2c}{c+1}-\frac{2}{c(c+1)}{c \choose 2}=\frac{2c}{c+1}-\frac{c-1}{c+1}=1,
$$
so both sides of the inequality are equal to $d$.
	
3) The bound is also sharp in situations similar to the following one. Let $n$ be the product of distinct primes $p_1,...,p_5$, $r=10$ and $\{d_i\}_{i=1}^{10}=\{p_{i_1}p_{i_2}p_{i_3}: 1\leq i_1<i_2<i_3\leq 5\}$. Let $s=5$. Since for any $p_j$ there are ${4 \choose 2}=6$ numbers $d_i$ divisible on $p_j$, the left-hand side is equal to $n$. Simple calculations show that so is the right-hand side. 
\end{rem}

\begin{proof}[\textbf{Proof of Lemma \ref{lem3.2}}]
	
Fix arbitrary $r\in\N$ and numbers $d_1,...,d_r$.  First, we write $d_i=\prod_pp^{\a_i(p)}$, and for each prime $p$  rearrange the numbers $\{\a_i(p): i\leq r\}$ in the nondecreasing order: let $\b_i=\b_i(p)$ be such that $\{\b_1,...,\b_r\}=\{\a_1,...,\a_r \}$ and $\b_1\leq ... \leq \b_r$. Fix $2\leq s\leq r$. It is easy to see that 
$$
\prod_{1\leq i_1<...<i_s\leq r}[d_{i_1},...,d_{i_s}]=\prod_pp^{{r-1\choose s-1 }\b_r+{r-2\choose s-1 }\b_{r-1}+...+{s-1\choose s-1 }\b_s}.
$$
Recall that 
$$
{r-1\choose s-1 }+{r-2\choose s-1 }+...+{s-1\choose s-1 }={r\choose s};
$$
in particular, it can be easily seen if we take all $d_i$ equal (say, equal to a prime $p$). We also have
\begin{equation*}\label{0.2}
\prod_{i=1}^rd_i=\prod_pp^{\sum_{i=1}^r\b_i}
\end{equation*}
and
\begin{equation*}\label{0.3}
\prod_{1\leq i<j\leq r}(d_i,d_j) =\prod_pp^{\b_{r-1}+2\b_{r-2}+...+(r-1)\b_1}=\prod_pp^{\sum_{i=1}^{r-1}(r-i)\b_i}.	
\end{equation*}
	
Denote for brevity $\g_i={i-1 \choose s-1}{r\choose s}^{-1}$; we saw that $\sum_{i=s}^r\g_i=1$. Let also $c=r-s+1$. To prove the lemma, we need to show that, for any prime $p$,
\begin{equation}\label{3.1}
L:=\g_r\b_r+...+\g_s\b_s \geq \frac{2}{c+1}\sum_{i=1}^r\b_i - \frac{2}{(c+1)c}\sum_{i=1}^r(r-i)\b_i=:R.
\end{equation}
We have
$$
R=\frac{2}{c+1}\sum_{i=1}^r\left(1-\frac{r-i}{r-s+1}\right)\b_i\leq \frac{2}{c+1}\sum_{i=s}^r\frac{i-s+1}{r-s+1}\b_i
=\d_s\b_s+...+\d_r\b_r.
$$
where $\d_i=\frac{2(i-s+1)}{c(c+1)}$. We claim that to prove (\ref{3.1}) it is enough to show that, for any $s\leq l\leq r$, 
\begin{equation}\label{3.2}
\sum_{i=s}^l\d_i \geq \sum_{i=s}^l\g_i.
\end{equation}
Indeed, once we know it, (\ref{3.1}) follows: since $\b_s\leq \b_{s+1}\leq ...\leq \b_r$ and $\sum_{i=s}^r\d_i=\sum_{i=s}^r\g_i=1$, we get
\begin{multline*}
R-L \leq \sum_{i=s}^r(\d_i-\g_i)\b_i \leq (\d_s+\d_{s+1}-\g_s-\g_{s+1})\b_{s+1}+\sum_{i=s+2}^r(\d_i-\g_i)\b_i\\
\leq \ldots \leq \left(\sum_{i=s}^r\d_i-\sum_{i=s}^r\g_i\right)\b_r=0. 	
\end{multline*}
Now we prove (\ref{3.2}). Fix $l$ with $s\leq l\leq r$, and denote $c_1=l-s+1$. We have 
$$
\sum_{i=s}^l\d_i=\frac{u}{c}\sum_{j=1}^{c_1}j=\frac{c_1(c_1+1)}{c(c+1)}=\frac{(l-s+1)(l-s+2)}{(r-s+1)(r-s+2)}
$$
and
$$\sum_{i=s}^l\g_i=\sum_{i=s}^l{i-1 \choose s-1}{r\choose s}^{-1}={l\choose s}{r \choose s}^{-1}=\frac{l!(r-s)!}{r!(l-s)!}.
$$
After a simple algebra we see that (\ref{3.2}) is equivalent to $r!(l-s+2)!\geq l!(r-s+2)!$, which is true since $r\geq l\geq s$. The claim follows. 
\end{proof}
	
Now we are ready to prove Theorem \ref{th1.4}. We shall show that if a number $n\leq 3Nk$ has $r\geq3$ divisors $2N\leq d_1<\ldots<d_r\leq 2N+2k$, then $k\geq c_0N^{\g_0-c_1/r}$. Since all the $d_i$ are at least $2N$ and all the $(d_i,d_j)$ are at most $2k$, we get from the previous lemma that, for any $1\leq c \leq r-1$,
$$
3Nk\geq m \geq [d_1,...,d_r] \geq  \frac{(2N)^{2r/(c+1)}}{(2k)^{r(r-1)/(c(c+1))}},
$$
or
$$
k^{c^2+r^2+c-r} \geq 3^{-O(r^2)}N^{2rc-c^2-c}.
$$
Now we choose $c=\lfloor\g_0r\rfloor$, so that $1\leq c\leq r-1$ (since $r\geq3$) and $c^2=\g_0^2r^2+O(r)$. Since $\g_0=(\sqrt5-1)/2$ is the $\arg\max\limits_{0<\a<1} \frac{2\a-\a^2}{1+\a^2}$ and the corresponding maximum is also equal to $\g_0$, we get 
$$
k\gg N^{\g_0-c_1/r}
$$
for some $c_1>0$, as desired. This completes the proof of Theorem \ref{th1.4}.

\section{Appendix: a Sidon subset of squares}

For completeness, here we provide the proof of the statement mentioned in the introduction.

\begin{prop}\label{prop5.1}
For any integer $N\geq1$, the set $\{n^2: N\leq n\leq N+2\sqrt2N^{1/2}\}$ is a Sidon set. 	
\end{prop}

\begin{proof} Suppose that for some $0\leq s_i\leq 2\sqrt2N^{1/2}$ we have	
$$
(N+s_1)^2+(N+s_2)^2= (N+s_3)^2+(N+s_4)^2;
$$
we need to show that $\{s_1,s_2\}=\{s_3,s_4\}$. Let us rewrite the equation as
\begin{equation}\label{5.1}
2N(s_1+s_2-s_3-s_4)=s_3^2+s_4^2-s_1^2-s_2^2, \quad 0\leq s_i\leq 2\sqrt2N^{1/2}.  
\end{equation}
We first show that this equality is impossible if $s_1+s_2\neq s_3+s_4$. Let $l=s_1+s_2-s_3-s_4\neq0$; without loss of generality we may suppose that $l>0$. Since the left-hand side is even, so is $s_3^2+s_4^2-s_1^2-s_2^2$; but $s_i^2\equiv s_i\pmod{2}$, and hence $l$ is also even. Thus $l\geq2$ and the left-hand side has absolute value at least $4N$. Now we want to get an upper bound for the right-hand side. We have
$$
s_1^2+s_2^2\geq0.5(s_1+s_2)^2=0.5(s_3+s_4+l)^2
$$
and therefore (by the upper bound for $s_i$)
$$
s_3^2+s_4^2-s_1^2-s_2^2 \leq 0.5(s_3-s_4)^2-(s_3+s_4)l-0.5l^2<0.5\max\{s_3^2,s_4^2\}\leq 4N.
$$
Thus the right-hand side is less than $4N$ and we get a contradiction.

Hence, only the four-tuples $(s_1,s_2,s_3,s_4)$ with $s_1+s_2=s_3+s_4$ matter. Such a four-tuple is a solution of (\ref{5.1}) if and only if $s_1^2+s_2^2=s_3^2+s_4^2$, which is equivalent (under the assumption $s_1+s_2=s_3+s_4$) to $s_1s_2=s_3s_4$. It follows that $\{s_1,s_2\}=\{s_3,s_4\}$, as desired.
\end{proof}	

We note that the term $2\sqrt2N^{1/2}$ in this proposition is essentially the best possible\footnote{We thank Alexander Kalmynin for this observation}, because for some integers $N$, the equation $a^2+b^2=c^2+d^2$ has nontrivial solutions in the interval $[N;N+2\sqrt2N^{1/2}+3]$. Indeed, one can take $N=2m^2-1$, where $m\geq2$ is an integer, and then $2m^2+4m+1=N+2\sqrt2\sqrt{N+1}+2\leq N+2\sqrt2N^{1/2}+3$ since $N\geq3$, and the claim follows from the identity
$$
(2m^2+4m+1)^2+(2m^2-1)^2=2(2m^2+2m+1)^2.
$$

\end{document}